\documentclass[11pt,twoside]{article}
\usepackage{latexsym}
\usepackage{amssymb,amsbsy,amsmath,amsfonts,amssymb,amscd,amsthm}
\usepackage{graphicx}
\usepackage{hyperref}
\usepackage{amsmath}
\usepackage{geometry}
\usepackage{color}
\usepackage{ulem}

\newcommand{\commentout}[1]{}

\newcommand{\R}{\mathbb{R}}
\newcommand{\N}{\mathbb{N}}

\newcommand{\beq}{\begin{equation}}
\newcommand{\beqa}{\begin{eqnarray}}
\newcommand{\bea} {\begin{array}{ll}}
\newcommand{\beqan}{\begin{eqnarray*}}
\newcommand{\eeq}{\end{equation}}
\newcommand{\eeqa}{\end{eqnarray}}
\newcommand{\eeqan}{\end{eqnarray*}}
\newcommand{\eea} {\end{array}}
\newtheorem{theorem}{Theorem}[section]
\newtheorem{lemma}[theorem]{Lemma}

\newtheorem{proposition}[theorem]{Proposition}

\newcommand{\cqfd}{{ \hfill
                       {\unskip\kern 6pt\penalty 500
                       \raise -2pt\hbox{\vrule\vbox to 6pt{\hrule width 6pt
                       \vfill\hrule}\vrule} \par}   }}

\title{\Large \bf Eikonal equations and pathwise solutions to fully non-linear SPDEs} 

\author{ Peter K. Friz$^{1,*}$, Paul Gassiat$^{2,*}$, Pierre-Louis Lions$^{3}$ and Panagiotis E. Souganidis$^{4,**}$}
\date{\today}

\begin{document}
\maketitle
\vspace*{1.0cm}
\pagenumbering{arabic}

\begin{abstract}

\noindent%

\noindent We study the existence and  uniqueness of the stochastic viscosity solutions  of fully nonlinear, possibly degenerate, second
order stochastic pde with quadratic Hamiltonians associated to a Riemannian geometry. The results are new and extend 
the class of equations studied so far by the last two authors.

\end{abstract}
\pagestyle{plain} \vspace*{1.0cm}

\noindent {\bf Key words.}  Fully non-linear stochastic partial differential equations; eikonal equations, pathwise stability, rough paths.
\\
\noindent {\bf AMS Class. Numbers.} 35R99, 60H15


\section{Introduction}

\noindent  The theory of stochastic viscosity solutions, %
 including existence, uniqueness and stability, developed by two of the authors (Lions and Souganidis \cite{MR1959710, MR1647162, MR1659958,
MR1799099, MR1807189, book})  is concerned with pathwise 
solutions to fully nonlinear, possibly degenerate, second order stochastic pde, which, in full generality, have the form
\begin{equation}\label{fully}
\begin{cases}
du=F(D^{2}u,Du, u, x, t) dt + \sum_{i=1}^{d}H_{i}\left( Du, u, x\right) d\xi ^{i} \ \text{ in } \ \R^N \times (0,T],\\[2mm]
u=u_0 \ \text{on} \ \R^N \times \{0\};
\end{cases}
\end{equation}
here  $F$  is degenerate elliptic and $\xi =(\xi^1,\cdots, \xi^d)$ is a continuous path. A particular example is a $d$-dimensional Brownian motion, in which case \eqref{fully} should be interpreted in the Stratonovich sense. 
Typically, 
$u \in \mathrm{BUC}(\mathbb{R}^N \times [0,T])$, the space of bounded uniformly continuous real-valued functions on $\mathbb{R}^N \times [0,T]$.
\smallskip

\noindent For the convenience of the reader we present a quick general overview of the theory:
The Lions--Souganidis theory applies to rather general paths when $H=H(p)$  and, as established in \cite{MR1659958, book}, there is a very precise trade off between the regularity of the paths and $H$. When $H=H(p,x)$ and $d=1$, the results of \cite{book} 
deal with general continuous, including Brownian paths, and 
the theory requires certain global structural conditions on $H$ involving higher order (up to three) derivatives in $x$ and $p$. Under similar conditions, Lions and Souganidis \cite{LS} have also established the wellposedness  of  \eqref{fully} for $d>1$ and Brownian paths. For completeness we note that, when $\xi$ is smooth, for example  $\text{C}^1$,   \eqref{fully}  
falls within the scope of the classical  Crandall-Lions viscosity
theory -- see, for example, Crandall, Ishii and Lions \cite{CIL}. 

\noindent The aforementioned conditions are used to control the length of the interval of existence of smooth solutions of the so-called doubled equation 
\begin{equation}\label{ takis double}
dw=(H(D_x w, x)-H(-D_y w,y))d\xi \ \text{ in } \  \R^N \times (t_0 - h_*, t_0+h_*)  
\end{equation}
with  initial datum 
\begin{equation}\label{ takis initial double}
w(x,y,t_0)=\lambda|x-y|^2
\end{equation}
as $\lambda \to \infty$ and uniformly for $|x-y|$ appropriately bounded.

\noindent It was, however, conjectured  in \cite{book} that, given a Hamiltonian $H$,  it may be possible to find initial data other than  $\lambda|x-y|^2$  for the doubled equation, which are better adapted to $H$, thus avoiding some of the growth conditions. As a matter of fact this was illustrated by an example 
when $N=1$. 
 
\smallskip

\noindent In this note we follow up on the remark above about the structural conditions on $H$ and identify a better suited  initial data for \eqref{ takis double} 
for the special class of quadratic Hamiltonians of  the form
\begin{equation}\label{takis quadratic}
H(p,x):=  (g^{-1}(x)p,p)= \sum_{i,j=1}^{N}g^{i,j}\left( x\right)  p_{i}p_{j},
\end{equation}
which are associated to a Riemannian geometry in $\R^N$ and  do not satisfy the conditions mentioned earlier, where 
\begin{equation}\label{takis geometry}
g=( g_{i,j})_{1\leq i,j\leq N} \in C^2(\R^N; {\mathcal{S}}^N)
\end{equation}
is positive definite, that is there exists $C>0$ such that, for all $w\in \R^N$,
\begin{equation}\label{takis positive}
 \frac{1}{C} \left\vert w\right\vert ^{2} \le \sum_{i,j} g_{i,j}\left( x\right) w^{i}w^{j} \le C \left\vert w\right\vert ^{2}.
\end{equation}
\noindent It follows from \eqref{takis quadratic} and \eqref{takis positive}  that $g$ is invertible and  $g^{-1}=( g^{i,j})_{1\leq i,j\leq N}\in C^2(\R^N; {\mathcal{S}}^N)$   is also positive definite;  here ${\mathcal{S}}^N$ is the space of $N\times N$-symmetric matrices and $(p,q)$ denotes the usual inner product of  the vectors $p$, $q$ $\in \mathbb{R}^N$.   
\noindent When dealing with \eqref{fully} it is necessary to strengthen \eqref{takis geometry} and we assume that 
\begin{equation}\label{takis bound}
 g, g^{-1} \in C^2_b(\R^N; {\mathcal{S}}^N), 
\end{equation}
where $ C^2_b(\R^N; {\mathcal{S}}^N)$ is the set of functions bounded in $C^2(\R^N; {\mathcal{S}}^N)$. Note that in this case \eqref{takis positive} is implied trivially.

\smallskip

%
%

%
\noindent 
The distance $d_{g}\left( x,y\right)$ with respect to $g$ of two points $x,y \in \mathbb{R}^N$ is given by
\begin{equation*}
d_{g}\left( x,y\right) :=  \inf \left\{ \int_{0}^{1} \frac{1}{2} (g(\gamma_t) \dot \gamma_t, \dot \gamma_t)^{1/2} dt :
\gamma \in C^1 ([0,1],{\mathbb{R}}^{N}), \gamma_0 = x, \gamma_1 = y \right\},
\end{equation*}%
and their associated  ``energy" is 
\begin{equation} \label{eq:DefE}
e_{g}( x,y) := 
d_{g}^2( x,y)=\inf \left\{ \int_{0}^{1} \frac{1}{4} (g(\gamma_t) \dot \gamma_t, \dot \gamma_t) dt :
\gamma \in C^1 ([0,1],{\mathbb{R}}^{N}), \gamma_0 = x, \gamma_1 = y \right\}.
\end{equation}%
Note that,  if $g=I$ the identity $N\times N$ matrix  in $\R^N$, then $d_I(x,y)=  \tfrac{1}{2}|x-y|$, the usual Euclidean distance,  and 
$e_{I}\left( x,y\right) = \tfrac{1}{4} |x-y|^2;$ more generally, (\ref{takis positive}) implies, with $C=c^2$ and for all $x, y\in {{\mathbb{R}}}^{N}$,
$$\frac{1}{2c}  |x-y| \le  d_{g}\left( x,y\right) \le \frac{1}{2} c |x-y|.$$
\smallskip


\noindent In addition,  we assume that 
\begin{equation}\label{asn:inj}
\text{ there exists $\Upsilon > 0$  such that $e_g \in C^1(\{ (x,y) \in {\R}^{N} \times \R^N: d_g(x,y) < \Upsilon \})$;}  
\end{equation} 
in the language of differential geometry  \eqref{asn:inj} is the same as to say that the manifold $( {\mathbb{R}} ^N,
g)$ has strictly positive injectivity radius. We remark that  \eqref{takis bound}  is sufficient for \eqref{asn:inj} (see, for example,  Proposition \ref{prop:Hess}), though (far) from necessary.


\smallskip


\noindent  We continue with some terminology and notation that we will need in the paper.  We write $I_N$ for the identity matrix in $\R^N$. A modulus is a nondecreasing, subadditive  function $\omega:[0,\infty) \to [0,\infty)$ such that  $\lim_{r\rightarrow 0}\omega \left( r\right) = \omega (0) = 0$.  We write $u\in \mathrm{UC_{g}}\left( {\mathbb{R}}^{N}\right) $ if $%
\left\vert u\left( x\right) -u\left( y\right) \right\vert \leq
\omega \left( d_{g}\left( x,y\right) \right) $ for some modulus $\omega$, 
and, given  $u\in \mathrm{UC_{g}}\left( {\mathbb{R}}^{N}\right)$, we denote by $\omega_{u}$ its modulus.
When $u$ is also bounded, we write $u \in \mathrm{BUC_g}\left( {\mathbb{R}}^{N}\right)$ and may take its
modulus bounded.
We denote by   $\mathrm{USC}$ (resp. $\mathrm{LSC}$)  the set of upper- (resp.
lower) semicontinuous functions in $\R^N$, and $ \mathrm{BUSC}$ (resp.  $ \mathrm{BLSC}$) is the set of bounded functions in  $\mathrm{USC}$ (resp.  $\mathrm{LSC}$). For a bounded continuous function $u:\R^k\to \R$, for some $k\in \N$,  and $A\subset \R^k$,  $\|u\|_{\infty, A}:=\sup_A |u|$. If $a,b\in\R$, then $a\wedge b:=\min (a,b)$,   $a_+:=\max(a,0)$ and $a_-:=\max(-a,0)$. Given a modulus $\omega$ and $\lambda >0$, we use the function $\theta: (0,\infty) \to (0,\infty)$ defined by 
\begin{equation}\label{takis theta}
\theta (\omega ,\lambda ):=\sup_{r\geq 0}\{\omega (r)-\lambda r^{2}/2\};
\end{equation}
and observe  that, in view of the assumed properties of the modulus,
\begin{equation}\label{takis theta 2}
\lim_{\lambda\to\infty} \theta(\omega;\lambda) = 0.
\end{equation}
\noindent Finally, for $k \in \N$,  $C^k_0 ([0,T];{\mathbb{R}}):=\{ \zeta \in C^k([0,T];{\mathbb{R}}): \zeta_0=0 \}$ and,  of for any two $\zeta, \xi \in C_0 ([0,T];{\mathbb{R}})$, 
we set
\begin{equation}\label{takis delta}
\Delta _{T}^{+}:=\max_{s\leq T}(\xi _{s}-\zeta _{s})\geq 0 \ \text{ and } \  \Delta
_{T}^{-}:=\max_{s\leq T}\{-(\xi _{s}-\zeta _{s})\}\geq 0.
\end{equation}%

\noindent We review next the approach taken  in \cite{MR1959710, MR1647162, MR1659958,
MR1799099, MR1807189, book} to define solutions to \eqref{fully}. The key idea is to show that the solutions of the initial value problems with smooth paths, which approximate locally uniformly the given continuous one, form a Cauchy family in
 $BUC(\R^N\times [0,T])$ for all $T>0$, and thus converge to a limit which is independent of the regularization. This limit is considered  as the  solution to \eqref{fully}. It follows that the  solution operator for \eqref{fully} is the extension in the class of continuous paths of the solution operator  for smooth paths. Then \cite{MR1959710, MR1647162, MR1659958,
MR1799099, MR1807189, book} introduced an intrinsic definition for a solution, called stochastic viscosity solution, which is satisfied by the uniform limit. Moreover, it was shown that the stochastic viscosity solutions satisfy a comparison principle and, hence, are intrinsically unique and can be constructed by the classical Perron's method (see \cite{book} and \cite{Seeger} for the complete argument).  The assumptions on the Hamiltonians mentioned above were used in these references 
to obtain both the Cauchy property and the intrinsic uniqueness. 
\smallskip

\noindent To prove the Cauchy property the aforementioned references consider the solutions to \eqref{fully} corresponding to two different smooth paths $\zeta_1$ and $\zeta_2$ and establish an upper bound for the $\sup$-norm of their difference. The classical viscosity theory provides immediately such a bound, which, however, depends on the $L^1$-norm of  $\dot \zeta_1 -\dot \zeta_2$. Such a bound is, of course, not useful since it blows up, as the paths approximate the given continuous path $\xi$.  
The novelty of the Lions-Souganidis theory is that it is possible to obtain far better control of the difference of the solutions based on the $\sup$-norm of $\zeta_1 -\zeta_2$ at the expense of some structural assumptions on $H$.  In the special case of  \eqref{fully} with $F=0$ and $H$ independent of $x$, a sharp estimate was obtained in \cite{book}. It was also remarked  there  that such bound cannot be expected to hold for spatially dependent Hamiltonians without additional restrictions. 
\smallskip

\noindent In this note we take advantage of the very particular quadratic structure of $H$ and obtain a local in time bound on the difference of two solutions with smooth paths. That the bound is local is due to the need to deal with smooth solutions of the Hamilton-Jacobi part of the equation. Quadratic Hamiltonians do not satisfy the assumptions in \cite{book}. Hence, the results here extend the class of  \eqref{fully} for which there exists a well posed solution. The bound obtained is also used to give an estimate for the  solutions to \eqref{fully},\eqref{takis quadratic} corresponding to different merely continuous paths as well as a modulus of continuity.
\smallskip

\noindent Next we present the results and begin with the comparison of solutions with smooth and different paths. Since the assumptions on the metric $g$ are slightly stronger  in the presence of the second order term in \eqref{fully}, we state two theorems. The first is for the first-order problem 
\begin{equation}\label{takis hamilton jacobi}
\begin{cases}
du - (g^{-1}(x) Du, Du)d\xi=0 \ \text {in } \ \R^N \times (0, T],\\[.05mm]
u=u_0 \ \text{on} \ \R^N \times \{0\},
\end{cases}
 \end{equation} 
and the second for \eqref{fully} with $H$ 
 given by \eqref{takis quadratic}.  Then we discuss the extension property and the comparison for general paths.

%


\smallskip

\noindent We first assume that we have smooth driving signals and estimate the difference of solutions. 
Since we are working with ``classical'' viscosity solutions, we write $u_t$ and $\dot \xi_t$ in place of of $du$ and $d\xi_t$.

\begin{theorem}\label{thm:1storder} Assume 
\eqref{takis geometry}, \eqref{takis positive} and \eqref{asn:inj} 
and let  $\xi ,\zeta \in C_0^{1}([0,T];{\mathbb{R}})$
and $u_{0},v_{0}\in \mathrm{BUC_g(\R^N)}$. 
If $u \in \mathrm{BUSC}(\R^N \times [0,T])$  and $v \in \mathrm{BLSC}(\R^N \times [0,T])$
are respectively viscosity sub- and super-solutions to
\begin{equation*}
u_t-(g^{-1}(x)Du, Du) \dot{\xi}\leq 0 \ \mathrm{ in } \ \R^N \times (0,T] \ \ \  u(\cdot,0 )\leq u_{0}  \  \mathrm{ on }  \ \R^N, 
\end{equation*} 
and
\begin{equation*}
v_t-(g^{-1}(x)Dv, Dv)\dot{\zeta}\geq 0 \ \mathrm{ in } \ \R^N \times (0,T] \ \ \ v(\cdot, 0 )\geq v_{0}  \  \mathrm{ on }  \ \R^N,
\end{equation*} 
then, if 
\begin{equation}
\Delta _{T}^{+}+\Delta _{T}^{-}\,<\frac{1}{2\left( \left\Vert u_0
\right\Vert_{\infty; {\mathbb{R}}^{N}}+\left\Vert v_0 \right\Vert_{\infty ; {%
\mathbb{R}}^{N}}\right) } \Upsilon^{2}, \label{AssumptionInjg}
\end{equation} 
\begin{equation}
\sup_{ \R^N \times [0,T]} \left( u-v\right) \;\leq \;\; \sup_{\R^N}\left( u_{0}-v_{0}\right) +\theta ( \omega _{u_{0}} \wedge \omega _{v_{0}},\frac{1%
}{\Delta _{T}^{+}}).  \label{eq:Compxizeta}
\end{equation}%
\end{theorem}

\smallskip

\noindent We consider now the second-order fully nonlinear
equation \eqref{fully} with quadratic Hamiltonians, that is
the initial value problem %
\begin{equation}\label{takis fully}
\left\{ 
\begin{array}{ll}
du=F\left( D^{2}u,Du,u,x,t\right) dt +(g^{-1}(x) Du, Du)d\xi & \ \mbox{ in } \ %
{\mathbb{R}}^{N} \times 0,T],\\[1mm] 
u(\cdot,0 )=u_{0}\in \mathrm{BUC}( {\mathbb{R}}^{N}), & 
\end{array}%
\right.
\end{equation}%
and introduce assumptions on $F$ in order to have a result similar to Theorem~\ref{thm:1storder}. 

\noindent In order to be able to have some checkable 
structural conditions on $F$, we find it necessary 
to replace \eqref{takis geometry} and \eqref{asn:inj} by the stronger conditions \eqref{takis bound} and 
\begin{equation}  
\label{Ad}
\text{there exists $\Upsilon >0$ such that}  \  D^2d_g^2 \  
\text{is bounded on } \{(x,y): d_g(x,y) <\Upsilon \}.
\end{equation}

\noindent As far as $F\in \text{ C} ( {\mathcal{S}}^N \times {\mathbb{R}} ^N \times {%
\mathbb{R}} ^N\times[0,T]; {\mathbb{R}} ) $ is concerned 
we assume that it is degenerate elliptic, that is for all $X, Y \in  \mathcal{S}^N$ and  $(p,r,x,t) \in {\mathbb{R}} ^N \times 
\mathbb{R}^N\times[0,T]$,
\begin{equation}  \label{deg}
\text{ $F(X,p,r,x,t)\leq F(Y,p,r,x,t)$ \  if \ $X\leq Y$;}
\end{equation}
Lipschitz continuous in $r$, that is 
\begin{equation}\label{takis lip}
\text{there exists $L>0$ such that $|F(X,p,r,x,t)-F(X,p,s,x,t)| \leq L|s-r|;$}
\end{equation}
bounded in $(x,t)$, in the sense that
\begin{equation}  \label{bounded}
\sup _{{\mathbb{R}} ^N\times[0,T]} |F(0,0,0,\cdot, \cdot)| <
\infty;
\end{equation}
and uniformly continuous for bounded $(X,p,r)$, that is, for any $R>0$, 
\begin{equation}  \label{uniform}
\text {$F$ is uniformly continuous on $M_{R} \times B_{R} \times [-R,R]
\times {\mathbb{R}} ^N \times [0,T]$},
\end{equation}
where $M_{R}$ and $B_{R}$ are respectively the balls of radius $R$ in ${\mathcal{S}}^N$ and ${\mathbb{R}}^{N}$.
\smallskip

\noindent Similarly to the classical theory of viscosity solutions, it is also necessary to assume something more about the joint continuity of $F$ in $X,p,x$, namely that
%
%
\begin{equation}  \label{thelemma}
\begin{cases}
\text{for each $R>0$ there exists a modulus $\omega_{F,R}$ such that, for all $\alpha, \varepsilon >0$ and uniformly on}\\
\text{ $t\in [0,T]$ and
$r\in [-R,R]$,} \\ 
F(X, \alpha D_{x}d_g^{2}(x,y), r,x,t)-F(Y,-\alpha
D_{y}d_g^{2}(x,y),r,x,t)\leq \omega_{F,R}(\alpha d_g^{2}(x,y)+d_g(x,y) + \varepsilon), \\ 
\text {whenever $d_g(x,y) < \Upsilon$ and $X,Y\in {\mathcal{S}}^N$ are
such that, for $A=D_{(x,y)}^{2}d_g^2(x,y)$, } \\[1mm] 
-(\alpha^2\varepsilon^{-1} + \|A\|) \left( 
\begin{array}{cc}
I\; & 0 \\ 
0\; & I%
\end{array}%
\right) \leq \left( 
\begin{array}{cc}
X & 0 \\ 
0 & -Y%
\end{array}%
\right) \leq \alpha A+\varepsilon A^{2}.\\[.5mm] 
\end{cases}%
\end{equation}
%
%
Note that in the deterministic theory the above assumption 
is stated using the Euclidean distance. Here it is convenient to use $d_g$ and as a result we find it necessary  to strengthen the 
assumptions on the metric $g$. 
\smallskip

\noindent To simplify the arguments below, instead of \eqref{takis lip}, we will assume that $F$ monotone in $r$, that is 
\begin{equation}  \label{monotone}
\text{there exists $\rho >0$ such that \ $F(X,p,r,x,t)-F(X,p,s,x,t) \geq
\rho (s-r)$ whenever $s\geq r$;}
\end{equation}
\noindent this is, of course, not a restriction since we can always consider the change $u(x,t)=e^{(L+\rho)t}v(x,t)$, which yields an equation for 
$v$ with a new $F$ satisfying \eqref{monotone} and path $\xi '$ such that $\dot \xi'_t= e^{(L+\rho)t}\dot \xi_t$.
\smallskip

\noindent To state the result we introduce some additional notation. 
For $\gamma >0$, we write
\begin{equation}
\tilde{\theta}(\omega ;\gamma ):=\sup_{r\geq 0}(\omega (r)-%
\frac{\gamma }{2}r),  \label{eq:defTheta}
\end{equation}
and, for $\xi , \zeta \in C^1([0,\infty);\R)$,
\begin{equation}\label{takis delta}
\Delta _{T}^{\gamma ,+}:=\sup_{t\leq T}\int_{0}^{t}e^{\gamma s}(\dot{\xi}%
_{s}-\dot{\zeta}_{s})ds \ \ \text{and} \  \ \Delta _{T}^{\gamma,-}:=\sup_{t\leq
T(}-\int_{0}^{t}e^{\gamma s}(\dot{\xi}_{s}-\dot{\zeta}_{s})ds).
\end{equation}
\noindent Finally, for  bounded $u_0, v_0:\R^N\to \R$, let 
\begin{equation*}
K:=\frac{2}{\rho }\sup_{(x,t) \in \R^N \times [0,T]}F(0,0,0,x,t)+\Vert u_{0}\Vert _{\infty ;{\mathbb{%
R}}^{N}}+\Vert v_{0}\Vert _{\infty ;{\mathbb{R}}^{N}}.
\end{equation*}

\noindent We have:
\begin{theorem}
\label{thm:2ndorder} Assume  \eqref{takis bound}, \eqref{Ad}, \eqref{deg},  \eqref{takis lip}, \eqref{bounded}, \eqref{uniform}  %
and \eqref{thelemma}, 
and let $\xi ,\zeta \in C_0^{1}([0,T];{\mathbb{R}})$ 
$u_{0},v_{0}\in \mathrm{BUC}({\mathbb{R}}^{N})
$ and $T>0$. If  $u$ $\in \mathrm{BUSC}( \R^N \times [0,T])$ and $v\in $ $\mathrm{BLSC}(  \R^N \times [0,T]) $ are respectively  viscosity sub- and super-solutions of  
\begin{equation}
u_t-F(D^{2}u,Du,u,x,t)-(g^{-1}(x) Du, Du)\dot{\xi}\leq 0\ \mbox{ in }{%
\mathbb{R}}^{N}\times (0,T]\qquad u(0,\cdot )\leq u_{0}\ \text{ on}\ {%
\mathbb{R}}^{N},  \label{eq:PDExi}
\end{equation}%
and 
\begin{equation}
v_t-F(D^{2}v,Dv,v,x,t)-(g^{-1}(x) Dv, Dv)\dot{\zeta}\geq 0\ \mbox{ in }{%
\mathbb{R}}^{N}\times (0,T]\qquad v(0,\cdot )\geq u_{0}\ \text{ on}\ {%
\mathbb{R}}^{N},  \label{eq:PDEzeta}
\end{equation}
then, if $\Gamma := \{ \gamma >0: \Delta _{T}^{\gamma ,+}+\Delta
_{T}^{\gamma,-}<\frac{\Upsilon^{2}}{4K} \ \text{and} \ \Delta
_{T}^{\gamma,-} <1\ \},$ 
\begin{equation}\label{takis zeta}
\begin{cases}
\sup_{ \R^N \times [0,T]} \left( u-v\right) \leq \sup_{\R^N} \left( u_{0}-v_{0}\right) _{+}
\\[1.5mm]
+\inf_{\gamma \in \Gamma }[ \theta ( \omega _{u_{0}}\wedge
\omega _{v_{0}},\frac{1}{\Delta _{T}^{\gamma ,+}}) +\frac{1}{\rho }%
\tilde{\theta}( \omega _{F,K};\gamma ) + \frac{1}{\rho} \omega
_{F,K}( 2 (K( \Delta _{T}^{\gamma ,+}+\Delta _{T}^{\gamma
;-}) )^{1/2} ].
\end{cases}
\end{equation}
\end{theorem}

\noindent Under their respective assumptions, Theorem~\ref{thm:1storder} and  Theorem~\ref{thm:2ndorder} imply that,
for paths $\xi\in C^1([0,\infty);\R)$ and $g\in \text{BUC}_g(\R^N)$, the initial value problems \eqref{takis hamilton jacobi} and \eqref{takis fully} have well-defined solution operators $$\mathcal{S}%
:(u_{0},\xi )\mapsto u\equiv \mathcal{S}^{\xi }\left[ u_{0}\right].$$

%

\noindent The main interest in the estimates \eqref{eq:Compxizeta} and \eqref{takis zeta}   is that they  provide a unique
continuous extension of this solution operator to all $\xi \in C([0,\infty);\R)$. Since the  proof is a simple reformulation of \eqref{eq:Compxizeta} and \eqref{takis zeta}, we omit it. 

\begin{theorem}
Under the assumptions of Theorem~\ref{thm:1storder} and  Theorem~\ref{thm:2ndorder}, the solution operator $$\mathcal{S}:BUC(\R^N) \times C^1([0,\infty);\R) \to BUC(\R^N \times [0,T])$$
admits a unique continuous extension to $$\bar{\mathcal{S}}:BUC(\R^N) \times C([0,\infty);\R) \to BUC(\R^N \times [0,T]).$$ 
In addition, there exists a nondecreasing $\Phi: [0,\infty )\rightarrow \lbrack 0,\infty ]$, depending only on the moduli and $\sup$-norms of 
$u_0, v_0 \in \mathrm{BUC_g} $, such that 
$\lim_{r\rightarrow 0}\Phi \left( r\right) = \Phi (0) = 0$, and,  for all $\xi ,\zeta \in C([0,T];{\mathbb{R}})$,
 \begin{equation} \label{eq:cont1storder}
\left\Vert \mathcal{S}^{\xi }\left[ u_{0}\right] -\mathcal{S}^{\zeta }\left[
v_{0}\right] \right\Vert _{\infty ;  {\mathbb{R}}^{N} \times [0,T]}\leq \Vert
u_{0}-v_{0}\Vert _{\infty ;{\mathbb{R}}^{N}}+\Phi \left( \left\Vert \xi
-\zeta \right\Vert _{\infty ;\left[ 0,T\right] }\right).
\end{equation}%

\end{theorem}

\smallskip



\noindent We also remark that for both problems  the proofs yield a, uniform in $t \in [0,T]$ and $\|\xi-\zeta\|_{\infty;[0,T]}$,  estimate 
for $u(x,t) - v(y,t)$. 
Applied to the solutions of \eqref{takis hamilton jacobi} and \eqref{takis fully}, this yields a (spatial) modulus of continuity which depends only on the initial datum,  $g$ and $F$ but not $\xi$. This allows to see (as in   \cite{DFO} and  \cite{MR1959710, MR1647162, MR1659958,
MR1799099, MR1807189, book}) that $\mathcal{S}$ and then $\bar{\mathcal{S}}$ indeed takes values in $BUC(\R^N \times [0,T])$.

%
%

%
%
\smallskip

\noindent An example of $F$ that satisfies the assumptions of Theorem~\ref{thm:2ndorder} is 
 the  
Hamilton-Jacobi-Isaacs operator 
\begin{equation}
F\left( M,p,r,x,t\right) = \inf_\alpha \sup_\beta \left\{ \mathrm{tr}\left( \sigma_{\alpha \beta} \sigma_{\alpha \beta} ^{T}\left( p,x\right) M\right)
+b_{\alpha \beta}\left( p,x\right) - c_{\alpha \beta} (x) r \right\},  \label{F_generic}
\end{equation}
with  
\begin{equation}\label{takis 1000}
\text{$\sigma, b, c$  bounded uniformly in $\alpha, \beta$}
\end{equation}
such that, for some modulus $\omega$ and 
constant $C>0$ and uniformly in $\alpha, \beta$, 
\begin{equation}\label{takis sigma1}
\left\vert \sigma_{\alpha \beta} (p,x)-\sigma_{\alpha \beta} (q,y)\right\vert \leq C(\left\vert
x-y\right\vert +\frac{\left\vert p-q\right\vert }{\left\vert p\right\vert
+\left\vert q\right\vert }),
\end{equation}%
and
\begin{equation}\label{takis sigma2}
\left\vert b_{\alpha \beta}(p,x)-b_{\alpha \beta}(q,y)\right\vert \leq \omega ((1+ |p|
+ |q|)|x-y|+ |p-q|),\;\;\;\;\;\;
\left| c_{\alpha \beta}(x) - c_{\alpha \beta}(y) \right| \leq \omega(|x-y|).
\end{equation}

\smallskip

\noindent The paper is organized as follows. In the next section we prove Theorem~\ref{thm:1storder}. Section 3 is about the proof of Theorem~\ref{thm:2ndorder}. In the last section we state and prove a result showing that \eqref{takis bound} implies \eqref{Ad} and verify that  \eqref{F_generic} satisfies the assumptions of Theorem~\ref{thm:2ndorder}. 

\section{The first order case: The proof of Theorem \protect\ref{thm:1storder}}

We begin by recalling without proof the basic  properties of the Riemannian energy $e_g$ which we need in this paper.
For more discussion we refer to, for example, \cite{MM} and the references therein.  

\begin{proposition} \label{prop:eik} 
Assume  \eqref{takis geometry}, \eqref{takis positive} and \eqref{asn:inj}. The  Riemannian energy $e_g$ defined by \eqref{eq:DefE} is (locally) absolutely continuous, almost everywhere  differentiable  and satisfies the Eikonal equations 
\begin{equation}\label{takis eikonal}
(g^{-1}(y)D_{y}e_g,  D_{y}e_g)=(g^{-1}(x) D_{x}e_g, D_{x}e_g)=e_g\left( x,y\right),
\end{equation}
on a subset $E$ of  $ \R^N \times \R^N$ of full measure. Moreover,  
$\left\{ (x,y) \in \R^N\times \R^N : d_g(x,y) < \Upsilon \right\} \subset E.$
\end{proposition}


\noindent The next lemma, which is  based on \eqref{takis eikonal} and the properties of $g$, is about  an observation 
which plays a vital role in the proofs. 
\smallskip

\noindent To this end, for $x,y \in {\mathbb{R}} ^N$, $\lambda
>0$ and $\xi,\zeta \in \text{C}_0^1([0,T])$, 
we set
\begin{equation}  \label{eq:phi}
\Phi^{\lambda} (x,y,t):= \frac{\lambda e_g(x,y)}{1-\lambda (\xi _{t}-\zeta
_{t})}.
\end{equation}

\begin{lemma}
\label{lem:PHI} Assume \eqref{takis geometry}, \eqref{takis positive} and \eqref{asn:inj} and  choose $\lambda <1/\Delta^+ _{T}$. Then 
\begin{equation}  \label{eq:IneqPhiL}
\frac{\lambda e_g}{1+\lambda \Delta^-_{T}} \leq \Phi^\lambda \leq \frac{\lambda e_g}{1-\lambda
\Delta^+_{T}}  \ \text{ on } \  \R^N \times \R^N \times [0,T]. %
\end{equation}
In addition,  in the set $\{ (x,y)\in {\mathbb{R}} ^N \times {\mathbb{R}} ^N: d_g(x,y)
< {\Upsilon} \}$, $\Phi^{\lambda}$ is a classical solution of 
\begin{equation}
w_t =(g^{-1}(x)D_x w, D_x w) \dot \xi - (g^{-1}(y)D_y w, D_y w)\dot \zeta.  \label{eq:PDEPHI}
\end{equation}
\end{lemma}

\begin{proof}
The first inequality is immediate from the  definition \eqref{takis delta} of $\Delta_T^{%
\pm}$. To prove \eqref{eq:PDEPHI}, we  observe that, in view of Proposition \ref{prop:eik}, 
we have 
\begin{align*}
\Phi_t^{\lambda }& =\frac{\lambda ^{2}e_g(x,y)}{(1-\lambda (\xi
_{t}-\zeta _{t}))^{2}}(\dot{\xi}_{t}-\dot{\zeta}_{t}), \\
(g^{-1}(x)D_x \Phi, D_x \Phi) & =\frac{ \lambda ^{2} }{(1-\lambda (\xi _{t}-\zeta
_{t}))^{2}} (g^{-1}(x)D_x e_g, D_x e_g)\\ 
&=\frac{ \lambda ^{2}e_g(x,y)}{(1-\lambda (\xi _{t}-\zeta
_{t}))^{2}} \\
&=\frac{ \lambda ^{2} }{(1-\lambda (\xi _{t}-\zeta
_{t}))^{2}} (g^{-1}(y)D_y e_g, D_y e_g) \;= (g^{-1}(y)D_y \Phi, D_y \Phi),
\end{align*}
and, hence, whenever $d_g( x,y)<\Upsilon$, the claim follows.
\end{proof}

\noindent The proof of Theorem~\ref{thm:1storder} follows the standard procedure
of doubling variables. The key idea introduced in \cite{MR1647162} is
to use special solutions of the Hamiltonian part of the equation as test
functions in all the comparison type-arguments, instead of the typical $%
\lambda |x-y|^2$ used in the ``deterministic'' viscosity theory. 
As already pointed out earlier, in the case of general Hamiltonians,  the construction of the test 
functions in \cite{MR1647162} is tedious and requires structural conditions on $H$. 
The special form of the problem at hand, however, yields easily such  tests functions, which are provided by Lemma \ref{lem:PHI}. 
\smallskip

\noindent \textit{Proof of Theorem \ref{thm:1storder}}. 
To prove %
\eqref{eq:Compxizeta} it suffices to show that, for all $\lambda $ in a
left-neighborhood of $(\Delta^+ _{T})^{-1}$, that is for $\lambda \in
((\Delta^+ _{T})^{-1}-\epsilon ,(\Delta^+ _{T})^{-1})$ for some $\epsilon >0$, and $x,y \in \R^N$ and $t\in [0,T]$, 
\begin{align}
u(x,t)-v(y,t)& \leq \Phi ^{\lambda }(x,y,t)+\sup_{x^{\prime },y^{\prime
} \in \R^N}\left( u_{0}(x^{\prime })-v_{0}(y^{\prime })-\lambda e_g(x^{\prime
},y^{\prime })\right)   \label{uxminusvy} \\
& \leq \Phi ^{\lambda }(x,y,t)+\sup_{\R^N}  \left(u_0-v_0\right)+\sup_{x^{\prime },y^{\prime }\in \R^N}\left( v_{0}(x^{\prime })-v_{0}(y^{\prime
})-\lambda e_g(x^{\prime },y^{\prime })\right) .  \notag
\end{align}%
\noindent Indeed taking $x=y$ in \eqref{uxminusvy} we find%
\begin{eqnarray*}
u(x,t)-v(x,t) &\leq &\sup_{\R^N}  \left(u_0-v_0\right)+\sup_{x^{\prime
},y^{\prime } \in \R^N}\left( v_{0}(x^{\prime })-v_{0}(y^{\prime })-\lambda
d^{2}(x^{\prime },y^{\prime })/2\right) _{+} \\
&\leq &\sup_{\R^N}  \left(u_0-v_0\right)+\sup_{r\geq 0}\left( \omega
_{v_{0}}\left( r\right) -\lambda r^{2}/2\right) _{+} \\
&=&\sup_{\R^N}  \left(u_0-v_0\right)+\theta \left( \omega
_{v_{0}},\lambda \right),
\end{eqnarray*}%
and  we conclude letting $\lambda \rightarrow (\Delta^+ _{T})^{-1}$.


\noindent  We begin with the observation that, since constants are solutions of \eqref{takis hamilton jacobi}, 
\begin{equation}  \label{eq:Bnd-uv-1}
u \;\leq \; \|u_0\|_{\infty; {\mathbb{R}} ^N}  \ \text{and} \  - v \; \leq \;
\|v_0\|_{\infty; {\mathbb{R}} ^N}. 
\end{equation}


\noindent Next we fix $\delta, \alpha >0$ and  $0< \lambda <(\Delta^+ _{T})^{-1}$ 
and consider the map 
$$(x,y,t) \to u(x,t)-v(y,t) -\Phi ^{\lambda }(x,y,t) -\delta \left(|x|^2 + |y|^2)\right)- \alpha t,$$
which, in view of \eqref{eq:Bnd-uv-1}, achieves its
maximum at some $(\hat x, \hat y, \hat t) \in {\mathbb{R}} ^N \times {%
\mathbb{R}} ^N \times [0,T]$ --note that below to keep the notation simple
we omit the dependence of $(\hat x, \hat y, \hat t)$ on $\lambda ,\delta,\alpha$.

\noindent Let%
\begin{align*}
M_{\lambda,\alpha,\delta}& :=\max_{{\mathbb{R}}^{N}\times {\mathbb{R}^{N}
\times [0,T]}}u(x,t)-v(y,t)-\Phi ^{\lambda }(x,y,t)-\delta
\left(|x|^2 + |y|^2)\right) - \alpha t \\
& =u(\hat{x},\hat{t})-v(\hat{y},\hat{t})-\Phi^{\lambda }(\hat x, \hat
y, \hat t) -\delta \left(|\hat x|^2 + |\hat y|^2\right)
 -\alpha \hat{t}.
\end{align*}%

\noindent The lemma below summarizes a number of important properties of $(\hat x,
\hat y, \hat t)$. Since the arguments in the proof are classical in the theory of viscosity solutions, see for example \cite{Barlesbook}, \cite{CIL}, we omit the details.

\begin{lemma}
\label{lem:misc} Suppose that the assumptions of Theorem~\ref{thm:1storder}  hold. Then: 
\begin{equation}  \label{lem1}
\begin{cases}
\text{(i) \  for any fixed $\lambda, \alpha >0$,  \ $\lim_{\delta \rightarrow 0} 
\delta (|\hat{x}|^2+|\hat{y}|^2)=0,$} \\[1mm] 
\
\text{(ii) }  \ e_g(\hat{x},\hat{y})\leq 2({1/\lambda +\Delta _{T}^{-}})(\|u\|
_{\infty }+ \| v\|_{\infty }), \\[1mm] 

\text{(iii)} \  \mbox{if }d_g(\hat x, \hat y) \leq \Upsilon, \text{ then }\;\; \ \\[1mm]
(g^{-1}(\hat x) D_{x}\Phi^{\lambda}( \hat x, \hat y, \hat t), D_{x}\Phi^{\lambda}( \hat x, \hat y, \hat t)) +\\[1mm]
 (g^{-1}(\hat y) D_{y}\Phi^{\lambda}(\hat x, \hat y, \hat t), D_{y}\Phi^{\lambda}( \hat x, \hat y, \hat t))
\leq 2\lambda (1-\lambda
\Delta^+_{T})^{-1} (\|u\| _{\infty }+ \| v\|_{\infty }),  \\[1.2mm] 
\text{(iv)} \ \lim_{\delta \rightarrow 0}
M_{\lambda,\alpha,\delta}=M_{\lambda,\alpha,0}.%
\end{cases}%
\end{equation}
\end{lemma}

\noindent 
Next we  argue
that, for any $\lambda $ in a sufficiently small left-neighborhood of $%
(\Delta^+ _{T})^{-1}$, we have $d_g(\hat{x},\hat{y})<\Upsilon$, which yields 
that the eikonal equation for $e$ are
valid at these points. 

\noindent In view of the bound on $d_g^{2}(\hat{x},\hat{y})=e_g(\hat{x},\hat{y})$ that follows from part (ii) of Lemma~\ref{lem:misc}, 
 it suffices to choose  
$\lambda $ so that 
\begin{equation*}
2({1/\lambda +\Delta _{T}^{-}})\left( \| u\| _{\infty
}+\| v\|_{\infty } \right) < \Upsilon^{2}.
\end{equation*}

\noindent Taking into account that we also need ${\Delta^+_{T}<1/\lambda }$, we are
led to the condition 
\begin{equation*}
{\Delta^+_{T}+\Delta _{T}^{-}<\frac{1}{\lambda} +\Delta _{T}^{-}}\leq \frac{1}{%
4\left( \| u\| _{\infty }+\| v\|_{\infty } 
\right) }\Upsilon^{2};
\end{equation*}%
and finding such $\lambda $ is possible in view of 
\eqref{AssumptionInjg}. 

\noindent If  $\hat{t}\in (0,T]$, we use the inequalities satisfied by $u$ and $v$ in the viscosity sense, noting that to simplify the  notation we omit the explicit
dependence of  derivatives of $\Phi$ on 
$(\hat x, \hat y, \hat t)$, and we find, in view of Lemma \ref{lem:PHI} 
and the  Cauchy-Schwarz's inequality, 
\begin{align*}
0 &\geq  \Phi_t^\lambda + \alpha - (g^{-1}(\hat x) (D_{ x}\Phi^\lambda + 2 \delta \hat x), (D_{ x}\Phi^\lambda + 2 \delta \hat x)) \dot{\xi}_{\hat t}
+ (g^{-1}(\hat y) (D_{ y}\Phi^\lambda - 2 \delta \hat y), (D_{ y}\Phi^\lambda - 2 \delta \hat y)) \dot{\zeta}_{\hat t}
\\
&\geq  \alpha - 
{\| \dot{\xi} \|}_{\infty;[0,T]} \left( 2 \delta (g^{-1}(\hat x) D_{ x} \Phi^{\lambda}, D_{ x} \Phi^\lambda)^{1/2} (g^{-1}(\hat x) \hat x, \hat x)^{1/2}
+ \delta^2 (g^{-1}(\hat x) \hat x, \hat x)   
\right) \\
& \;\;\;\;\;\;\;\;\;\;\;\;\;\;\;\;- {\| \dot{\zeta} \|} _{\infty;[0,T]}
\left( 2 \delta (g^{-1}(\hat y) D_{ x} \Phi^{\lambda}, D_{ y} \Phi^{\lambda})^{1/2} (g^{-1}(\hat y) \hat y, \hat y)^{1/2}
+ \delta^2 (g^{-1}(\hat y) \hat y, \hat y)  \right).
\end{align*}

\noindent Using  again Lemma \ref{lem:misc} (i)-(iii),   we can now let $\delta $ $\rightarrow $ $0$
to obtain 
$\alpha\leq 0$, which is a contradiction.

\noindent
It follows that, for all $\delta$ small enough, we must have  $\hat t = 0$ and, hence,
$$M_{\lambda,\alpha,\delta }\leq \left( u_{0}(%
\hat{x})-v_{0}(\hat{y})-\lambda e(\hat{x},\hat{y})\right)\leq \sup_{\R^N}\left( u_{0}-v_{0}\right) +\theta \left( \omega _{u_{0}} \wedge \omega _{v_{0}},\lambda\right).$$
\noindent Letting first $\delta \to 0$ and then $\alpha \to 0$, concludes the proof of \eqref{uxminusvy}.
\hfill$\qed$

\section{The second-order case: The proof of Theorem~\ref{thm:2ndorder}}

\indent Since the proof of Theorem~\ref{thm:2ndorder} is in many places very similar to that of 
Theorem \ref{thm:1storder}, we omit arguments that follow along straightforward modifications. 

\noindent In the next lemma we introduce the
modified test functions, which here will depend on an additional  parameter $%
\gamma$ corresponding to a time exponential. Since its proof is similar to the one of Lemma~\ref{lem:PHI}, we omit it.

\begin{lemma}
\label{lem:eikonalPDE2} Fix $T, \lambda >0, \gamma \geq 0,\,\xi ,\zeta \in
C^{1}([ 0,T]; {\mathbb{R}} ^N) $ with $\xi _{0}=\zeta _{0}=0$ and assume that
$\lambda \Delta_T^{\gamma;+} <1.$ Then
\begin{equation*}
\Phi ^{\lambda ,\gamma }(x,y,t):=\frac{\lambda e^{\gamma t}}{1-\lambda
\int_{0}^{t}e^{\gamma s}(\dot{\xi}_{s}-\dot{\zeta}_{s})ds}e_g(x,y)
\end{equation*}
is a classical solution, in $ \left\{(x,y) \in {\mathbb{R}} ^N\times {\mathbb{R}} ^N: d\left(
x,y\right) <\Upsilon\right\} \times [0,T]$, of 
\begin{equation*}
w_t - \gamma w- (g^{-1}(x) D_{x}w, D_{x}w)\dot{\xi} + (g^{-1}(y) D_{y}w, D_{y}w)\dot{\zeta}=0.
\end{equation*}
\end{lemma}

\noindent  Next we specify the range of $\lambda$'s we will use. We set 
\begin{equation}\label{takis lambda}
\bar{\lambda} := (\Delta^{\gamma, +}_T)^{-1} \ \text{and} \  \underline{\lambda }:= \frac{4K}{\Upsilon - 4K\Delta^{\gamma,-}_T},
\end{equation}
and observe that, in view of our assumptions, we have  $\bar{\lambda} >  \underline{\lambda }.$
We say that $\lambda$ is admissible for fixed $\gamma$ and $\alpha$, if $\lambda \in (\underline{\lambda }, \bar{\lambda})$.


\noindent Also note that, if $u$, $v$, $u_{0},v_{0}, \xi$, $\zeta$ and $F$ are as in the statement of Theorem~\ref{thm:2ndorder}, then 
\begin{equation} \label{eq:Bnd-uv}
\sup_{\mathbb{R}^N \times [0,T]} (u - v) \leq K.
\end{equation}
%
\noindent  For fixed $\delta >0$ and  $ \lambda $  admissible 
we consider the map 
$$(x,y,t) \to u(x,t)-v(y,t) -\Phi ^{\lambda,\gamma }(x,y,t) -\delta \left(|x|^2 + |y|^2)\right),$$
which, in view of \eqref{eq:Bnd-uv-1}, achieves its
maximum at some $(\hat x, \hat y, \hat t) \in {\mathbb{R}} ^N \times {%
\mathbb{R}} ^N \times [0,T]$ --as before to keep the notation simple
we omit the dependence of $(\hat x, \hat y, \hat t)$ on $\lambda ,\delta$.

\noindent Let
\begin{eqnarray}
M_{\lambda ,\gamma ,\delta } &:=& \max_{\left( x,y, t\right) \in 
 {\mathbb{R}}^{N}\times {\mathbb{R}}^{N}\times [0,T]}u(x,t)-v(y,t)-\Phi
^{\lambda ,\gamma }(x,y,t)-\delta (|x|^{2}+|y|^{2})  \label{DefMlgd} \\
&= &u(\hat{x},\hat{t})-v(\hat{y},\hat{t})-\Phi ^{\lambda ,\gamma }(
\hat{x},\hat{y},\hat{t})-\delta (|\hat{x}|^{2}+|\hat{y}|^{2}). \notag
\end{eqnarray}%
%
\noindent The following claim is the analogue of  Lemma \ref%
{lem:misc}. As before when writing $\Phi$ and its derivatives we omit their arguments.
\begin{lemma}
\label{lem:miscnew}  Under the assumptions of Theorem~\ref{thm:2ndorder} and for  $\lambda$ admissible we have:%
\begin{equation*}
\begin{cases}
 \text{(i)} \  \lim_{\delta \rightarrow 0}\delta (|\hat{x}|^{2}+|\hat{y}|^{2})=0, \ \ \text{(ii)} \  e_g(\hat{x},\hat{y})\leq 2 K ( \frac{1}{\lambda}+\Delta^{\gamma,-}_T ), \  \ 
\text{(iii)} \  |D_{{x}}\Phi^{\lambda ,\gamma }|^{2}\leq 
\frac{2 \lambda e^{\gamma T}}{1-\lambda \Delta^{\gamma;+}_T} K \\
  \text{and} \ \ \text{(iv)}  \ \lim_{\delta \rightarrow 0}M_{\lambda ,\gamma ,\delta
}=M_{\lambda ,\gamma ,0}.
\end{cases}
\end{equation*}
\end{lemma}


\noindent \textit{Proof of Theorem \ref{thm:2ndorder}}. If, for some sequence  $\delta \to 0$, 
$\hat{t}=0$, then 
\begin{equation}\label{Estimate_when_t_hat_zero}
M_{\lambda ,\gamma ,0}=\lim_{\delta \to 0} M_{\lambda ,\gamma ,\delta}
\leq u_{0}(\hat{x})-v_{0}(\hat{y})-\Phi ^{\lambda ,\gamma }(\hat{x}, \hat{y},0)
\leq \| \left( u_{0}-v_{0}\right) _{+}\| _{\infty}+\theta ( \omega _{u_{0}}\wedge \omega _{v_{0}},\lambda) .
\end{equation}
\noindent We now treat the case where $\hat{t}\in (0,T]$ for all $\delta$ small enough.
%
%
%

\noindent  Since, in view of  Lemma~\ref{lem:miscnew}(ii) and the assumptions (recalling that $\lambda$ is admissible), the test-function $\Phi^{\lambda, \gamma}$ is smooth at $(\hat x, \hat y, \hat t )$, it follows from the theory of viscosity solutions (see, for example, \cite{CIL})  that
\begin{eqnarray}\label{eq:visc}
0 \ge {\Phi}^{\lambda, \gamma}_t-F(X+2\delta I, D_{x}{\Phi}^{\lambda, \gamma}
+2\delta \hat{x}, u(\hat{x},\hat{t}), \hat x, \hat t)-(g^{-1}(\hat{x})(D_{{x}}{\Phi}^{\lambda, \gamma}
+2\delta \hat{x}), D_{x}{\Phi}^{\lambda, \gamma}
+2\delta \hat{x}) \dot{\xi}_{\hat{t}} &&  \notag \\
+F(Y-2\delta I, -D_{y}{\Phi}^{\lambda, \gamma}
-2\delta \hat{y}, v(\hat{y}, \hat t), \hat y, \hat t) +
(g^{-1}(\hat{y})(D_{y}{\Phi}^{\lambda, \gamma}
+2\delta \hat{y}), D_{y}{\Phi}^{\lambda, \gamma}
+2\delta \hat{y})\dot \zeta_{\hat t},
%
\end{eqnarray}%
where $X, Y \in \mathcal {S}^N$ are such that 
%
%
for a given $\varepsilon >0$,
\begin{equation}  \label{eq:XY}
-\left( \frac{\hat{\alpha}^2}{\varepsilon}+\hat{\alpha}|D^{2}e_g(\hat{x},%
\hat{y})|\right) \left( 
\begin{array}{cc}
I\; & 0 \\ 
0\; & I%
\end{array}%
\right) \leq \;\;\left( 
\begin{array}{cc}
X & 0 \\ 
0 & -Y%
\end{array}%
\right) \leq \hat{\alpha}D^{2}e_g(\hat{x},\hat{y})+\varepsilon%
(D^{2}e_g(\hat{x},\hat{y}))^{2}
\end{equation}%
%
%
and 
\begin{equation}
\hat{\alpha}:=\frac{\lambda e^{\gamma \hat t}}{1-\lambda \int_{0}^{\hat{t}}e^{\gamma s}(\dot{\xi}_{s}-\dot{\zeta}_{s})ds}=\frac{{\Phi}^{\lambda, \gamma}(\hat x, \hat y, \hat t)}{e_g( \hat{x},\hat{y})}.
\label{alpha}
\end{equation}%
\smallskip

\noindent Then, as in the usual proof of the comparison of viscosity solutions, combining \eqref{eq:visc} and
\eqref{monotone}, we get that 
\begin{equation} \label{ takis vscosity}
\rho (u(\hat x, \hat{t})-v(\hat{y}, \hat{t}))_+ \leq (a)+(b)+(c)+(d),
\end{equation}
where 
\begin{equation}\label{takis a}
(a):=-F(X,  D_{x}\Phi^{\lambda, \gamma}, u(
\hat{x}, \hat{t}),  \hat x, \hat{t})+F(X+2\delta, D_{x}\Phi^{\lambda, \gamma}+2\delta \hat{x}, u(\hat{x}, \hat{t}), \hat x, \hat t),
\end{equation}
\begin{equation}\label{takis b}
(b):=F(Y, - D_{y}\Phi^{\lambda, \gamma}, v(
\hat y, \hat{t}),  \hat x, \hat{t})-F(Y+2\delta
I, - D_{y}\Phi^{\lambda, \gamma} -2\delta \hat{y}, v(
\hat y, \hat{t}),  \hat x, \hat{t}),
\end{equation}
\begin{equation}\label{takis c}
(c):= \begin{cases}
\Phi^{\lambda, \gamma}_t+\gamma \Phi^{\lambda, \gamma}+(g^{-1}(\hat x) (D_{{x}}{
\Phi}^{\lambda, \gamma} +2\delta \hat{x}), D_{{x}}{%
\Phi}^{\lambda, \gamma} +2\delta \hat{x})\dot{\xi}_{\hat {t}}\\[1mm]
- (g^{-1}(\hat y) (D_{{y}}{%
\Phi}^{\lambda, \gamma}-2\delta \hat{y}), D_{{y}}{%
\Phi}^{\lambda, \gamma}-2\delta \hat{y})\dot \zeta_{\hat t},
\end{cases}
\end{equation}
and
\begin{equation}\label{takis s}
(d):=-\gamma \Phi^{\lambda, \gamma}+F( X,  D_{x}\Phi^{\lambda, \gamma}, u(
\hat{x}, \hat{t}),  \hat x, \hat{t})
-F(Y, - D_{y}\Phi^{\lambda, \gamma}, v(
\hat y, \hat{t}), \hat y, \hat t).
\end{equation}
\noindent Since %
\eqref{Ad} and  \eqref{eq:XY} imply that $X$ and $Y$ stay bounded, in view of \eqref{uniform}, 
we get $\limsup_{\delta \to 0}((a)+(b))=0$. 
\smallskip

\noindent Moreover, the quadratic form of the equation satisfied by 
$\Phi^{\lambda, \gamma} =\hat \alpha e_g$ gives
\begin{equation*}
(c) \leq C \delta |D_{\hat x}\Phi^{\lambda, \gamma} | ( |\hat x| + |\hat
y|) \left( {\| \dot{\xi} \|}_{\infty;[0,T]} + {\| \dot{\zeta} \|}%
_{\infty;[0,T]}\right),
\end{equation*}
and using Lemma \ref{lem:miscnew} (i),(iii) we find
$\lim_{\delta \rightarrow 0}(c)=0$. 

\noindent For the last
term, note that Lemma \ref{lem:miscnew} (ii) and \eqref{thelemma} yield, always at the point $(\hat x, \hat y, \hat t)$,
%
%
\begin{eqnarray*}\label{takis d}
(d) &= &-\gamma \hat{\alpha}{e_g}+ F( X,  \hat \alpha D_{x} e_g, u(
\hat{x}, \hat{t}),  \hat x, \hat{t})
-F(Y, -\hat \alpha D_{y} e_g, v(\hat y, \hat{t}), \hat y, \hat t)\\
& \leq & -\frac{\gamma \hat{\alpha}}{2}d^{2}( \hat{x},\hat{y})
+\omega _{F, K}\left( d\left( \hat{x},\hat{y}\right) +\hat{\alpha}d^{2}\left( 
\hat{x},\hat{y}\right) + \varepsilon \right) \\
&\leq &-\frac{\gamma \hat{\alpha}}{2}d^{2}\left( \hat{x},\hat{y}\right)
+\omega _{F, K}(\hat{\alpha}d^{2}( \hat{x},\hat{y})) +
\omega_{F,K} (d( \hat{x},\hat{y})  + \omega _{F, K}\left( \varepsilon \right)\\
&\leq &\tilde{\theta}\left( \omega _{F,K},\gamma \right) + \omega_{F, K} (2 
(K (\frac{1}{\lambda} + \Delta^{\gamma;-})^{1/2}) + \omega _{F, K}\left( \varepsilon \right).
\end{eqnarray*}
%
%

\noindent Combining the last four estimates and \eqref{Estimate_when_t_hat_zero}  
and letting $\varepsilon \to 0$
 we find 
that, for all $\lambda \in (\underline{\lambda},\bar{\lambda})$ 
\begin{eqnarray*}
u(x,t)-v(x, t) &\leq & M_{\lambda ,\gamma ,0}= \lim_{\delta
\to 0}M_{\lambda ,\gamma ,\delta } \\
&\leq & \left\Vert \left( u_{0}-v_{0}\right) _{+}\right\Vert _{\infty
}+ \theta\left( \omega _{u_{0}}\wedge \omega _{v_{0}},\lambda \right) + 
\frac{1}{\rho }\tilde{\theta}\left( \omega _{F,K},\gamma \right) + \frac{1}{\rho} \omega_{F,K}
(2 (K (\frac{1}{\lambda} + \Delta^{\gamma;-}_T)^{1/2}) .
\end{eqnarray*}%

\noindent Letting $\lambda \to \bar \lambda$ 
and using the continuity of $\theta$ in the last argument,  we finally obtain that, for all $\gamma \in \Gamma $,
\begin{eqnarray*}
u-v &\leq &\left\Vert \left( u_{0}-v_{0}\right) _{+}\right\Vert _{\infty
}+\theta \left( \omega _{u_{0}}\wedge \omega _{v_{0}},
(\Delta^{\gamma;+}_T)^{-1} \right) +\frac{1}{\rho }\tilde{\theta} \left(
\omega_{F,K}; \gamma \right) +  \frac{1}{\rho} \omega_{F,K}(2 (K (\frac{1}{\lambda} + \Delta^{\gamma;-}_T)^{1/2}).
\end{eqnarray*}
\hfill\qed

\section{The properties of the geodesic energy and the assumptions of Theorem~\ref{thm:2ndorder}}

\noindent In this section we prove that $C^2_b$-bounds on $g$ and $g^{-1}$
imply \eqref{Ad} 
 and verify that the $F$ in \eqref{F_generic}, if \eqref{takis sigma1} and \eqref{takis sigma2} hold, satisfies the assumptions of Theorem \ref{thm:2ndorder}.

\noindent We begin with the former. 
\begin{proposition} \label{prop:g} 
Assume \eqref{takis bound}. 
Then there exists $\Upsilon >0$ such that, in the set $\left\{ (x,y): d_{g}(x,y)<\Upsilon \right\}$, 
$e_g$ is twice continuously differentiable and  \eqref{Ad} is satisfied with bounds
depending only on appropriate norms of $g,g^{-1}$.
\end{proposition}

\begin{proof}
We begin by recalling  some basic facts concerning geodesics and distances. 

\noindent For each fixed point $%
x$, there is a unique geodesic with starting velocity $v=\dot{\gamma}\left(
0\right) $ given by  $\gamma _{t}= X_{t}$,  where $\left( X,P\right)_{t\geq 0} $
is the solution to the characteristic equations
\begin{equation}\label{eq:HJ}
\begin{cases}
\dot{X}_{s}= 2g^{-1}(X_s) P_s \;\;\;\;\;X_{0}=x, \\ 
\dot{P}_{s}=- (Dg^{-1}(X_s)P_s,P_s)\;\;\;\;\;P_{0}=p= \tfrac{1}{2} g(x)v.
%
\end{cases}
\end{equation}

\noindent Equivalently $(\gamma)_{t \geq 0} $ satisfies the second order system of ode %
\begin{equation}
\ddot{\gamma}_{t}+\Gamma _{ij}^{k}\left( \gamma _{t}\right) \dot{\gamma}%
_{t}^{i}\dot{\gamma}_{t}^{j}=0 \ \ \  \gamma _{0}=x, \ \dot{\gamma}_{0}=v \label{geoODE}
\end{equation}%
with 
\begin{equation}   \label{Chris}
\Gamma _{ij}^{k}:=g^{k\ell }\left( \partial _{i}g_{\ell
j}+\partial _{j}g_{\ell i}-\partial _{\ell }g_{ij}\right) \text{.}
\end{equation}
\noindent It is easy to see that \eqref{eq:HJ}  has a global solution $\left( X,P\right)_{t\geq 0}$, since, in view of  \eqref{takis positive} and \eqref{takis bound} as well as  the invariance of the flow, 
we have, for $t\geq 0$, 
\begin{equation}    \label{Psmall_if_p0small}
| P_t |
\approx  (g^{-1}(X_t)P_t, P_t)  =  (g^{-1}(X_0)P_0, P_0)  \approx |p|^2.  
\end{equation}%
%

\noindent As a consequence, the projected end-point map
$ E_{x}\left( p\right) :=X_{1}( x, p) $ 
is well-defined for any $p$. 
%
\smallskip

\noindent We note that the energy along a geodesic $\gamma$
emerging from $\gamma_0 = x$ has a simple expression in terms of $p=P_0$ or $v=%
\dot{\gamma}_{0}$. Indeed, invariance of the Hamilonian $H(x,p) = (g^{-1}(x)p,p)$ under the flow yields 
\begin{equation} \label{gamma_0squ}
|\dot{\gamma}_{0}|^2_g := ( g\left( x\right)
v,v)  =  4  \, H(x,p) 
=  4 \int_{0}^{1}  H(X_t,P_t) dt= 
\int_{0}^{1}(g\left( \gamma _{t}\right) \dot{\gamma}%
_{t},\dot{\gamma}_{t}) dt.
\end{equation}%

\noindent  
It is a basic fact that distance minimizing curves (geodesics)
are also energy minimizing.
\noindent Indeed, given $x,y$, (\ref{eq:HJ}) and equivalently (\ref{geoODE}), are the
first-order optimality necessary conditions for these minimization problems. Hence, in view of (\ref{gamma_0squ}),
\begin{align}
e_{g}\left( x,y\right) & =   \inf \left\{ \frac{1}{4}\left\vert \dot{\gamma}%
_{0}\right\vert _{g}^{2}:\gamma \text{ satisfies (\ref{geoODE}) } \text{ with} \  \gamma
_{0}=x \text{ and } \gamma _{1}=y\right\}  \nonumber \\
& =   \inf \left\{  H(X_0,P_0) : (X,P) \text{ satisfies (\ref{eq:HJ}) } \text{ with} \  
X_{0}=x \text{ and } X_{1}=y\right\} . \nonumber
\end{align}%


\noindent A standard compactness argument implies the existence of at least one geodesic connecting two given points $x,y$. In general, however, more than one geodesic from $x$ to $y$ may exist, each determined by its initial velocity $%
\dot{\gamma}_{0}=v$ or equivalently $P_0 = p = \frac{1}{2} g(x) v$, upon departure from $x$. 

\noindent  It turns out that, for $y$ close to $x$, there exists
exactly one geodesic. 
 Indeed, if $g^{-1} \in C^2$, it is clear from (\ref{eq:HJ}) that $E_x = (p \mapsto X_1(x,p))$ has $C^1$-dependence in $p$. Since $D_{p} E_{x} (p)$ is non-degenerate in a
neighborhood of $p=0$, 
it follows from the inverse function theorem that, for $y$ close enough to $x$, 
one can solve $y=E_x(p)$ uniquely for $p=E_{x}^{-1}(y)$ with $C^1$-dependence in $y$. Hence\begin{equation}
e_g (x,y) = H (x, p) = H(x, E_{x}^{-1}(y)). \label{energyHamiltonian}
\end{equation}
The gradient $D_{x}e_g\left( x,y\right) $
points in the direction of maximal increase of $x\mapsto e\left( x,y\right) $.  Since $E_{x}^{-1}\left( y\right) =p$ is precisely the co-velocity of $%
X$ at $X_{0}=x$ and $X$ is the geodesic from $x$ to $X_{1}=y$, it follows that 
\begin{equation}
\label{eq:DxE}
D_{x}e_g\left( x,y\right) =-E_{x}^{-1}y=-p.
\end{equation}%
This easily implies that, for points $x,y$ close enough, the energy has continuous second derivatives. Indeed, existence of continuous mixed derivatives $D^2_{xy} e_g$ follows immediately from the $C^1$-regularity of $E_{x}^{-1}=E_{x}^{-1}(y)$. Concerning $D^2_{xx} e_g$ (and by symmetry $D^2_{yy}e_g$) we set $p=p_0$ above and note that
by exchanging the roles of $x,y$, we have $D_{y}e_g\left( x,y\right) =-E_{y}^{-1}\left(
x\right) =P_{1}\left( x,p_{0}\right) =:p_{1}$ and so, from (\ref{eq:HJ}),
\begin{equation} \label{p1minusp0}
D_{y}e_g\left( x,y\right) +D_{x}e_g\left( x,y\right)  = p_{1}-p_{0}=-\int_{0}^{1}H_{x}\left( X_{s},P_{s}\right) ds.
\end{equation}%
The existence and continuity of $D^2_{xx} e_g$ is then clear, since the right-hand side above has $C^1$-dependence in $x$ as is immediate from (\ref{eq:HJ}) and $g^{-1} \in C^2$.

\noindent  Since all the above considerations have been so far local, it is necessary to address the (global)
question of regularity in a strip around the diagonal $\{ x = y \}$. 
For this  we need to control $\alpha_1:=D_p E_x (p) = D_p X_1 (x,p)$. We do this by considering  the tangent flow
$$
   (\alpha_t, \beta_t) := (D_p X_t(x,p),D_p P_t(x,p)),
$$   
which solves the matrix-valued linear ode
\begin{equation} \label{eq:Jac}
\begin{cases}
\dot{\alpha}_{s}=2Dg^{-1}(X_s) P_s \alpha_{s}+ 2g^{-1}(X_s)\beta _{s},\;\;\;\;\;\alpha
_{0}=0, \\[1mm] 
\dot{\beta}_{s}=- (D^2g^{-1}(X_s)P_s, P_s) \alpha
_{s}-  2 \, Dg^{-1}(X_s) P_s   \beta _{s},\;\;\;\;\;\beta
_{0}=I  .%
\end{cases}
\end{equation}%
We now argue that, uniformly in $x$, 
$$D_p E_x (p) = \alpha_1 \approx  2 \, g^{-1} (x).$$ 
It follows that $D_p E_x$ is non-degenerate, again
uniformly in $x$. Indeed if $\alpha_0 = 0$, whenever $p$ is small, $X_\cdot \approx x, \beta_\cdot \approx I$ and we have
\[
\left( 
\begin{array}{c}
\dot{\alpha} \\ 
\dot{\beta}%
\end{array}%
\right) =\left( 
\begin{array}{ll}
\text{(small)} & 2 g^{-1} (X_\cdot)  \\ 
\text{(small)} & \, \, \text{(small)}%
\end{array}%
\right) \left( 
\begin{array}{c}
\alpha  \\ 
\beta 
\end{array}%
\right) \, .
\]
\noindent  Next we prove the above claim. First, it follows from (\ref{Psmall_if_p0small}) that if $p$ is small, then $P_t(x,p)$ stays small, over, for example,  a unit time interval $[0,1]$. Moreover, since   
$\dot{X} =  2  \, g^{-1} (X) P$, the  boundedness of $g^{-1}$ yields that the path $X_\cdot(x, p)$ also stays close to $X_0 = x$, again uniformly on $[0,1]$.
Furthermore, the $C^1$- and  $C^2$-bounds on $g^{-1}$ yield that the matrices  $Dg^{-1}(X_s) P_s$ and  $ (D^2g^{-1}(X_s)P_s, P_s)$ are small  along $(X_{s},P_{s})$, while $2g^{-1}(X_s)$ is plainly 
bounded. %
\noindent This implies that $\alpha_\cdot $ and $\beta_\cdot $ will stay
bounded. Then $\dot{\beta}_\cdot$ will be small, and, hence, $\beta_\cdot$ will be close to $\beta_0 = I$, uniformly on $[0,1]$. 
In turn, $\dot\alpha_s$ is the sum of a small term plus $ 2 \ g^{-1}(X_s)\beta _{s}   \approx  2 g^{-1} (x)$. %
In other words, 
$$\dot{\alpha}_s =  2 g^{-1}(x) + \Theta_s(x,p),$$
where
$$\lim_{\delta \to 0} \sup_{|p|\leq \delta} \sup_{s \in [0,1]} \sup_{x \in \R^N} \left|\Theta_s(x,p) \right|= 0.$$
Since  $D_p E_x (p) = \alpha_1 =  \int_0^1 \dot\alpha_s ds$, it follows that, there exists some $\delta>0$, which can be taken proportional to $M^{-4}$ where $M=1+ \| g \|_{C^{2}}
+\| g^{-1}\|_{C^{2}}$, such that%

\begin{equation*}
 | D_p E_x (p) -   2 g^{-1} (x)  |  \le \|g\|_\infty^{-1} \mbox{ for all } x,p \in \R^N, |p| \leq \delta.
\end{equation*}
 Note that the choice of the constant $ \|g\|_\infty^{-1}>0$ on the right-hand side guarantees that $D_p E_x (p)$ remains non-degenerate, uniformly in $x$.

\noindent It follows by the inverse function theorem  
that $p \mapsto E_{x}(p)$ is a diffeomorphism from $B_\delta$ onto a neighbourhood of $x$. We claim that this neighbourhood contains
a ball of radius $\Upsilon>0$, which may be taken independent of $x$.

\noindent For this we observe that, with $p = E_x^{-1}(y)$, (\ref{energyHamiltonian}) yields
$$d_g(x,y) = \sqrt{e_g (x, y)} = \sqrt{ H (x, p) } = \sqrt{ (g^{-1}(x) p, p) } .$$
Hence it suffices to choose $\Upsilon>0$ small enough so that $(g^{-1}(x) p,p) \leq \Upsilon^2$ implies
$\left\vert p \right\vert \leq \delta$, an obvious choice being $\Upsilon = \delta /c$, where $c^2$ is the ellipticity constant of $g$. 

\noindent At last, we note that \eqref{eq:DxE}, in conjunction with the just obtained
quantitative bounds, implies that $D^{2}e_{g}$ is
bounded on $\left\{ (x,y)\in \R^N\times \R^N: d_{g}(x,y)<\Upsilon \right\}$. 
 Indeed, with $p=E_x^{-1}(y)$, we have $-D^2_{xy} e_g(x,y) = D_y E_x^{-1}(y) = (D_p E_x(p))^{-1} \approx \frac{1}{2} g(x)$ 
which readily leads to bounds of the second mixed derivatives, uniformly over $x,y$ of distance at most $\Upsilon$. Similar uniform bounds
for $D^2_{xx} e_g$ (and then $D^2_{yy} e_g$) are obtained by differentiating (\ref{p1minusp0}) with respect to $x$ and estimating the resulting 
right-hand side. 
\end{proof}

\noindent The comparison proofs in the viscosity theory typically employ
quadratic penalty function $\phi \left( x,y\right) =\frac{1}{2}\left\vert
x-y\right\vert ^{2}$ and make use of (trivial) identities such as%
\begin{equation*}
D_{x}\phi +D_{y}\phi =(x-y)+(y-x)=0
\end{equation*}%
and %
\begin{equation*}
\left( 
\begin{array}{cc}
D^2_{xx}\phi & D^2_{xy}\phi \\ 
D^2_{yx}\phi & D^2_{yy}\phi%
\end{array}%
\right) = \left( 
\begin{array}{cc}
I & -I \\ 
-I & I%
\end{array}%
\right) , \,\, \,
\,\,\left( p,\,q\right) \left( 
\begin{array}{cc}
D^2_{xx}\phi & D^2_{xy}\phi \\ 
D^2_{yx}\phi & D^2_{yy}\phi%
\end{array}%
\right) \left( 
\begin{array}{c}
p \\ 
q%
\end{array}%
\right) \leq \left\vert p-q\right\vert ^{2}\text{.}
\end{equation*}%

\noindent To see what one can expect in more general settings, consider first the
case of $g$ obtained from the Euclidean metric, written in different
coordinates, say $x=\Psi ^{-1}\left( \tilde{x}\right) $, in which  case, we  have%
\begin{equation*}
e_g\left( x,y\right) =\left\vert \Psi \left( x\right) -\Psi \left(
y\right) \right\vert ^{2}.
\end{equation*}%
\noindent If $\Psi $ is bounded in $C^{2}$, it is immediate that%
\begin{eqnarray*}
D_{x}e+D_{y}e &=&(\Psi \left( x\right) -\Psi \left( y\right) )D\Psi \left(
x\right) +(\Psi \left( y\right) -\Psi \left( x\right) )D\Psi \left( y\right)
\\
&=&(\Psi \left( x\right) -\Psi \left( y\right) \left( D\Psi \left( y\right)
-D\Psi \left( y\right) \right),
\end{eqnarray*}
and, hence, 
\begin{equation}
\left\vert D_{x}e+D_{y}e\right\vert \lesssim \left\vert x-y\right\vert
^{2}
\end{equation}%
and, similarly, 
\begin{equation*}
\left( p,\,q\right) \left( 
\begin{array}{cc}
D^2_{xx}e & D^2_{xy}e \\ 
D^2_{yx}e & D^2_{yy}e%
\end{array}%
\right) \left( 
\begin{array}{c}
p \\ 
q%
\end{array}%
\right) \lesssim \left\vert p-q\right\vert ^{2}+\left\vert x-y\right\vert
^{2}.
\end{equation*}%

\noindent Unfortunately no such arguments work in the case of general Riemannian
metric, since, in general, there is no change of variables of the form $\left\vert \Psi \left(
x\right) -\Psi \left( y\right) \right\vert $ that reduces $d_{g}\left(
x,y\right) $ to a Euclidean distance. 
\smallskip

\noindent The next two propositions provide estimates that can be used in the comparison proofs in place of the exact identities above.

\begin{proposition} \label{prop:DxePlusEye} Assume \eqref{takis bound}. 
Then there exists $\Upsilon>0$ such that whenever $d_g(x,y) < \Upsilon$,
\begin{equation} \label{eq:exey}
\left\vert D_{x}e_g+D_{y}e_g\right\vert \le L \left\vert x-y\right\vert ^{2}
\end{equation}%
with a constant $L$ that depends only on the $C^{1}$-bounds of $g,g^{-1}$. 
\end{proposition}

\begin{proof}  As pointed out in the proof of Proposition~\ref {prop:g}, for all $(x,y): d_{g}(x,y)<\Upsilon$,
\begin{equation*}
D_{y}e_g\left( x,y\right) +D_{x}e_g\left( x,y\right)  = -\int_{0}^{1} ( Dg^{-1}(X_{s})P_{s},P_{s} ) ds.
\end{equation*}%
Using that $g^{-1}\in C^{1}$ and the fact that $g$ is uniformly comparable to the Euclidean
metric we get 
\begin{eqnarray*}
\left\vert D_{y}e\left( x,y\right) +D_{x}e\left( x,y\right) \right\vert 
&\leq &\|g^{-1}\| _{C^{1}}\int_{0}^{1}\left\vert
P_{s}\right\vert ^{2}ds \\
&\leq & \| g^{-1} \| _{C^{1}} \| g\| _{\infty
}\int_{0}^{1}( g^{-1}\left( X_{s}\right) P_{s},P_{s})
ds
\end{eqnarray*}%
and, hence,  thanks to invariance of the Hamiltonian under the flow and (\ref{energyHamiltonian}), 
\begin{equation*}
\left\vert D_{y}e\left( x,y\right) +D_{x}e\left( x,y\right) \right\vert
\lesssim \int_{0}^{1}H(X_{s},P_{s})ds=H\left( X_{0},P_{0}\right) = e_g \left(
x,y\right) .
\end{equation*}%
Using again that $g$ is uniformly comparable to the Euclidean metric the proof is finished. 

\end{proof}

\noindent
The following claim applies  in particular under condition \eqref{takis bound}, which, in view of  Propostion \ref{prop:g},  implies $C^2$-regularity of the energy near the diagonal. The proof is based on an argument in a forthcoming paper by the last two authors \cite{LS2016}.

\begin{proposition} \label{prop:Hess}
Assume there exists $\Upsilon >0$ such that, in the set $\left\{ (x,y): d_{g}(x,y)<\Upsilon \right\}$, 
$e_g$ is twice continuously differentiable. Then, whenever $d_g(x,y) < \Upsilon$,

\begin{equation*}
\left( 
\begin{array}{cc}
D_{xx}^{2}e_{g} & D_{xy}^{2}e_{g} \\ 
D_{yx}^{2}e_{g} & D_{yy}^{2}e_{g}%
\end{array}%
\right) \leq L\left( 
\begin{array}{cc}
I_{N} & -I_{N} \\ 
-I_{N} & I_{N}%
\end{array}%
\right) +L\left\vert x-y\right\vert ^{2}I_{2N}
\end{equation*}%
with a constant $L$ that only depends on the $C^{2}$-bounds of $g,g^{-1}$.
\end{proposition}

\begin{proof}
In view of the assumed $C^{2}$-regularity of the energy, we find that, as $\varepsilon\to 0$, 
\begin{equation*}
\begin{cases}
\varepsilon ^{2}\left( \left( p,q\right) \left( 
\begin{array}{cc}
D_{xx}^{2}e_{g} & D_{xy}^{2}e_{g} \\ 
D_{yx}^{2}e_{g} & D_{yy}^{2}e_{g}%
\end{array}%
\right) \left( 
\begin{array}{c}
p \\ 
q%
\end{array}%
\right) \right)  \\[1mm]
=e_{g}\left( x+\varepsilon p,y+\varepsilon q\right)
+e_{g}\left( x-\varepsilon p,y-\varepsilon w\right) -2e_{g}\left( x,y\right) 
\end{cases}
\end{equation*}
We estimate the second-order difference on the right-hand side, keeping $%
v=\varepsilon p,w=\varepsilon q$ fixed. 

\noindent Let $(\gamma_t)_{t\in [0,1]}$
be a geodesic connecting $x,y$,
parametrized at constant speed (in the metric $g$) so that%
\begin{equation}  \label{speed_est}
\left\vert \dot{\gamma}_{t}\right\vert \leq c\left( g\left( \gamma
_{t}\right) \dot{\gamma}_{t},\dot{\gamma}_{t}\right) ^{1/2}=
2cd_{g}\left( x,y\right) \leq C\left\vert x-y\right\vert ,
\end{equation}
and, for $\Delta :=w-v$, define the paths  
\begin{equation*}
\gamma _{t}^{\pm } : =\gamma _{t}\pm \left( v+t\Delta \right) 
\end{equation*}%
which connect $x\pm v$ to  $y\pm w$. 
\noindent Then%
\begin{eqnarray*}
&&e\left( x+v,y+w\right) +e\left( x-v,y-w\right) -2e\left( x,y\right)  \\
&\leq &\int_{0}^{1}\frac{1}{4}\left( g\left( \gamma _{t}^{+}\right) \dot{%
\gamma}_{t}^{+},\dot{\gamma}_{t}^{+}\right) dt+\int_{0}^{1}\frac{1}{4}\left(
g\left( \gamma _{t}^{-}\right) \dot{\gamma}_{t}^{-},\dot{\gamma}%
_{t}^{-}\right) dt-2\int_{0}^{1}\frac{1}{4}\left( g\left( \gamma _{t}\right) 
\dot{\gamma}_{t},\dot{\gamma}_{t}\right) dt \\
&=&\frac{1}{4}\int_{0}^{1}\left[ \left( g\left( \gamma _{t}^{+}\right)
\left( \dot{\gamma}_{t}+\Delta \right) ,\dot{\gamma}_{t}+\Delta \right)
+\left( g\left( \gamma _{t}^{-}\right) \left( \dot{\gamma}_{t}-\Delta
\right) ,\dot{\gamma}_{t}-\Delta \right) dt-2\left( g\left( \gamma
_{t}\right) \dot{\gamma}_{t},\dot{\gamma}_{t}\right) \right] dt
\end{eqnarray*}%
Using the $C^{2}$-regularity of $g$, writing $\delta =v+t\Delta $, and noting that 
that $\left\vert \delta \right\vert \leq \left\vert v\right\vert +\left\vert
w\right\vert $, we find %
\begin{equation*}
g\left( \gamma _{t}^{\pm }\right) =g\left( \gamma _{t}\right) \pm \left(
Dg\left( \gamma _{t}\right) ,\delta \right) +\frac{1}{2}\left( D^{2}g\left(
\gamma _{t}\right) \delta ,\delta \right) +o\left( \left\vert v\right\vert
^{2}+\left\vert w\right\vert ^{2}\right) .
\end{equation*}%
Collecting terms (in $g,Dg,D^{2}g$) then leads to%
\begin{equation*}
e\left( x+v,y+w\right) +e\left( x-v,y-w\right) -2e\left( x,y\right) \leq 
\frac{1}{4}\int_{0}^{1}\left[ \left( i\right) +\left( ii\right) +\left(
iii\right) +\left( E\right) \right] dt
\end{equation*}%
where 
\begin{equation*}
\begin{cases}
\left( i\right):=2\left( g\left( \gamma _{t}\right) \Delta ,\Delta
\right) \\[1mm]
\left( ii\right):=4\left( \left( Dg\left( \gamma _{t}\right) ,\delta
\right) \dot{\gamma}_{t},\Delta \right)\\[1mm]
 \left( iii\right):=\left( \left( D^{2}g\left( \gamma _{t}\right) \delta
,\delta \right) \dot{\gamma}_{t},\dot{\gamma}_{t}\right) +\left( \left(
D^{2}g\left( \gamma _{t}\right) \delta ,\delta \right) \Delta ,\Delta
\right) 
 \end{cases}
 \end{equation*}
\noindent It is immediate that,  
\begin{equation*}
\begin{cases}
\left\vert \left( i\right) \right\vert \leq 2\left\Vert
g\right\Vert _{\infty }\left\vert w-v\right\vert ^{2}\\[1mm]
\left\vert \left( ii\right) \right\vert \leq 4C\left\Vert
Dg\right\Vert _{\infty }\left( \left\vert v\right\vert +\left\vert
w\right\vert \right) \left\vert x-y\right\vert \left\vert w-v\right\vert \\[1mm]
\left\vert iii\right\vert \leq \left\Vert D^{2}g\right\Vert
_{\infty }\left( \left\vert v\right\vert ^{2}+\left\vert w\right\vert
^{2}\right) \left( C^{2}\left\vert x-y\right\vert ^{2}+\left\vert
w-v\right\vert ^{2}\right).
\end{cases}
\end{equation*}
\noindent Moreover, expanding $g$, as $v,w\rightarrow 0$,  we find 
\begin{equation*}
\left( E\right) =(\left\vert \dot{\gamma}_{t}\right\vert ^{2}+\left\vert
\Delta \right\vert ^{2}) o\left( \left\vert v\right\vert
^{2}+\left\vert w\right\vert ^{2}\right) =(C^{2}\left\vert x-y\right\vert
^{2}+\left\vert w-v\right\vert ^{2}) o\left( \left\vert v\right\vert
^{2}+\left\vert w\right\vert ^{2}\right).
\end{equation*}

\noindent With $v=\varepsilon p,w=\varepsilon q$ all terms above are of order $O\left(
\varepsilon ^{2}\right) $, with the exception of the second  term contributing to $\left(
iii\right) $ and the error term $\left( E\right) $ which are actually $o\left(
\varepsilon ^{2}\right) $, and hence negligible as $\varepsilon \rightarrow 0
$. 

\noindent Indeed 
\begin{eqnarray*}
&&e( x+\varepsilon p,y+\varepsilon q) +e( x-\varepsilon
p,y-\varepsilon w) -2e( x,y)   \\
 &\leq&\frac{\varepsilon ^{2}}{4}[ 2\Vert g\Vert _{\infty
}\vert p-q \vert ^{2}+4C\Vert Dg\Vert _{\infty }(
\vert p\vert +\vert q\vert) \vert
x-y\vert \vert p-q\vert  \\
&&+\Vert D^{2}g\Vert _{\infty }( \vert
p\vert ^{2}+\vert q\vert ^{2}) ( C^{2}\vert
x-y\vert ^{2}+O( \varepsilon ^{2})) +(\vert
x-y\vert ^{2}+O( \varepsilon ^{2}) )o( 1)] .
\end{eqnarray*}%
\noindent Using again the   Cauchy- Schwarz inequality to handle the middle term in the
estimate, we find, for some  $K>0$ that only depends on $g$, 
\begin{eqnarray*}
&&\left( p,q\right) \left( 
\begin{array}{cc}
D_{xx}^{2}e_{g} & D_{xy}^{2}e_{g} \\ 
D_{yx}^{2}e_{g} & D_{yy}^{2}e_{g}%
\end{array}%
\right) \left( 
\begin{array}{c}
p \\ 
q%
\end{array}%
\right)  \\
&\leq & [ \frac{1}{2}\left\Vert g\right\Vert _{\infty }\left\vert
p-q\right\vert ^{2}+C\left\Vert Dg\right\Vert _{\infty }\left( \left\vert
p\right\vert +\left\vert q\right\vert \right) \left\vert x-y\right\vert
\left\vert p-q\right\vert +\frac{C^{2}}{4}\left\Vert D^{2}g\right\Vert _{\infty
}( \left\vert p\right\vert ^{2}+\left\vert q\right\vert ^{2})
\left\vert x-y\right\vert ^{2}]  \\
&\leq &K\left\vert p-q\right\vert ^{2}+K\left\vert x-y\right\vert ^{2}\left(
\left\vert p\right\vert ^{2}+\left\vert q^{2}\right\vert \right).
\end{eqnarray*}

\end{proof}

\noindent We conclude by checking that the assumptions of Theorem~\ref{thm:2ndorder} are satisfied for $F$ of Hamilton-Jacobi-Isaacs type. Note that condition \eqref{takis bound} is valid, as demanded by that theorem. 

\begin{proposition}
Let $F$ be given by \eqref{F_generic}, satisfying \eqref{takis sigma1} and \eqref{takis sigma2}. Then $F$ satisfies \eqref{deg}-\eqref{thelemma}.
\end{proposition}

\begin{proof}
 Since it is clear that all assumptions  are stable under taking sup and inf, we will only treat the quasilinear case
$$F(M,p,r,x) = \operatorname{Tr}(a(p,x)M) + b(p,x) - c(x) r,$$ 
and we concentrate on  \eqref{thelemma}, since the others are obvious.

\noindent
\noindent With $t,x,y,r,\alpha,X,Y$ taken as in the
statement of the assumption, we have 
\begin{eqnarray*}
&&F(X, \alpha D_{x}d_g^{2}(x,y), r,x,t)-F(Y,-\alpha D_{y}d_g^{2}(x,y),r,x,t)
\\
&=& b\left(p,x \right) - b\left(q,y \right) + \mathrm{Tr}\left(\sigma
\sigma^T(p,x) X - \sigma \sigma^T(q,y) Y\right) + (c(x)-c(y))r,
\end{eqnarray*}
where for simplicity we write  $p: =\alpha D_{x}d_g^{2}(x,y)$ and $q:=-\alpha
D_{y}d_g^{2}(x,y)$. 

\noindent Noting  that Proposition \ref{prop:DxePlusEye} yields 
\begin{equation*}
|p-q| \leq \alpha K |x-y|^2,
\end{equation*}
and using the eikonal equation for $e_g = d_g^2$, for some $C>0$ we find  
\begin{equation*}
\alpha C^{-1} |x-y| \leq |p| + |q| \leq \alpha C |x-y|.
\end{equation*}
\noindent The assumptions on $b$ also give 
\begin{eqnarray*}
b\left(p,x \right) - b\left(q,y\right) &\leq& \omega\left( \left(1 + |p| +
|q| \right) |x-y| +|p-q|\right) \\
&\leq& \omega\left( \left(1+C \alpha |x-y|\right) |x-y| + K \alpha |x-y|^2
\right).
\end{eqnarray*}
\noindent For the second order term, we use Proposition \ref{prop:Hess} and get 
%
%
\begin{eqnarray*}
\mathrm{Tr}\left(\sigma \sigma^T(p,x) X - \sigma \sigma^T(q,y) Y\right) &=& 
(\sigma(p,x)  , \,  \sigma(q,y))%
\begin{pmatrix}
X & 0 \\ 
0 & -Y%
\end{pmatrix}
\begin{pmatrix}
\sigma(p,x) \\ 
\sigma(q,y)%
\end{pmatrix}
\\
&\leq& \alpha \left( \sigma(p,x), \,  \sigma(q,y)\right) \left( 
\begin{array}{cc}
D^2_{xx}e & D^2_{xy}e \\ 
D^2_{yx}e & D^2_{yy}e%
\end{array}%
\right) \left( 
\begin{array}{c}
\sigma(p,x) \\ 
\sigma(q,y)%
\end{array}%
\right) \\
&& + \;\; \varepsilon \left\| \left( 
\begin{array}{cc}
D^2_{xx}e & D^2_{xy}e \\ 
D^2_{yx}e & D^2_{yy}e%
\end{array}%
\right) \left( 
\begin{array}{c}
\sigma(p,x) \\ 
\sigma(q,y)%
\end{array}%
\right) \right\|^2 \\
&\leq& K \alpha ( \left\vert\sigma(p,x) - \sigma(q,y)\right\vert
^{2}+\left\vert x-y\right\vert ^{2}) + K \varepsilon\\
&\leq& K^{\prime }\alpha\left( |x-y|^2 + \frac{|p-q|^2}{(|p|+|q|)^2} \right) + K^{\prime } \varepsilon
\\
&\leq& K^{\prime \prime} \left( \alpha |x-y|^2 + \varepsilon\right).
\end{eqnarray*}
%

\end{proof}

%
%

\bigskip
\bigskip

\noindent ($^{1}$)  Technische Universit\"at Berlin and Weierstra\ss --Institut f\"ur Angewandte Analysis und Stochastik\\
Stra\ss e des 17. Juni 136 \\
10623 Berlin \\
email: friz@math.tu-berlin.de
\\ \\
\noindent ($^{2}$) CEREMADE, Universit\'e de Paris-Dauphine, \\
Place du Mar\'echal-de-Lattre-de-Tassigny \\
75775 Paris cedex 16, France \\
email: gassiat@ceremade.dauphine.fr
\\ \\ 
\noindent ($^{3}$)  Coll{\`e}ge de France and CEREMADE, Universit{\'e} de Paris-Dauphine, \\
1, Place Marcellin Berthelot \\
75005 Paris Cedex 5, France \\
email: lions@ceremade.dauphine.fr
\\ \\
($^{4}$)  Department of Mathematics \\
             University of Chicago \\
             Chicago, IL 60637, USA \\
            email: souganidis@math.uchicago.edu
\\ \\
($^{*}$)  Partially supported by the European Research Council under the European Union’s Seventh Framework Programme (FP7/2007-2013)
/ ERC grant agreement nr. 258237
\\ \\
($^{**}$)  Partially supported by the National Science Foundation grant DMS-1266383.

\end{document}